\documentclass[3p,times]{amsart}

\usepackage{amssymb}

\usepackage{amsthm}

\usepackage{enumerate} 

\usepackage{amsmath}
\usepackage{amssymb}

\newtheorem{theorem}{Theorem}[section] 
\newtheorem{lemma}[theorem]{Lemma}   

\newtheorem{proposition}[theorem]{Proposition}

\newtheorem{main-theorem}{Theorem}

\newtheorem*{problem*}{Problem}

\theoremstyle{definition}

\newtheorem*{question*}{Question}

\renewcommand{\mod}{\operatorname{mod}}

\newcommand{\End}{\operatorname{End}}

\newcommand{\Hom}{\operatorname{Hom}}

\newcommand{\T}{\operatorname{T}}

\newcommand{\op}{\operatorname{op}}

\newcommand{\add}{\operatorname{add}}

\newcommand{\bR}{\mathbb{R}}

\newcommand{\bZ}{\mathbb{Z}}

\makeatletter
\newcommand{\vect}[1]{%
  \vbox{\m@th \ialign {##\crcr
  \vectfill\crcr\noalign{\kern-\p@ \nointerlineskip}
  $\hfil\displaystyle{#1}\hfil$\crcr}}}
\def\vectfill{%
  $\m@th\smash-\mkern-7mu%
  \cleaders\hbox{$\mkern-2mu\smash-\mkern-2mu$}\hfill
  \mkern-7mu\raisebox{-3.81pt}[\p@][\p@]{$\mathord\mathchar"017E$}$}

\newcommand{\amsvect}{%
  \mathpalette {\overarrow@\vectfill@}}
\def\vectfill@{\arrowfill@\relbar\relbar{\raisebox{-3.81pt}[\p@][\p@]{$\mathord\mathchar"017E$}}}

\newcommand{\amsvectb}{%
  \mathpalette {\overarrow@\vectfillb@}}
\newcommand{\vecbar}{%
  \scalebox{0.8}{$\relbar$}}
\def\vectfillb@{\arrowfill@\vecbar\vecbar{\raisebox{-4pt}[\p@][\p@]{$\mathord\mathchar"017E$}}}
\makeatother
\usepackage{etex} 

\usepackage{etex} 

\usepackage{dsfont}         
\usepackage[h]{esvect}

\usepackage[matrix,arrow,curve,frame,arc]{xy}

\usepackage{pgf}
\usepackage{tikz}

\usepackage{color}

\usetikzlibrary{arrows,shapes,automata,backgrounds,petri,calc}

\newcommand{\tikzAngleOfLine}{\tikz@AngleOfLine}
\def\tikz@AngleOfLine(#1)(#2)#3{%
\pgfmathanglebetweenpoints{%
\pgfpointanchor{#1}{center}}{%
\pgfpointanchor{#2}{center}}
\pgfmathsetmacro{#3}{\pgfmathresult}%
}

\begin{document}

\title{Higher spherical algebras}


\author[K. Edrmann]{Karin Erdmann}
\address[Karin Erdmann]{Mathematical Institute,
   Oxford University,
   ROQ, Oxford OX2 6GG,
   United Kingdom}
\email{erdmann@maths.ox.ac.uk}

\author[A. Skowro\'nski]{Andrzej Skowro\'nski}
\address[Andrzej Skowro\'nski]{Faculty of Mathematics and Computer Science,
   Nicolaus Copernicus University,
   Chopina~12/18,
   87-100 Toru\'n,
   Poland}
\email{skowron@mat.uni.torun.pl}

\begin{abstract}
We introduce and study  higher spherical algebras,
an exotic family of finite-dimensional algebras over an
algebraically closed field.
We prove that every such an algebra
is derived equivalent to a higher tetrahedral algebra
studied in \cite{ES3}, and
hence that it is a tame symmetric periodic algebra
of period $4$.
%
%
%
%
\end{abstract}


\maketitle

%
%

\vspace*{-2cm}

\section{Introduction and main results}\label{sec:intro}

Throughout this paper, $K$ will denote a fixed algebraically closed field.
By an algebra we mean an associative finite-dimensional $K$-algebra
with an identity.
For an algebra $A$, we denote by $\mod A$ the category of
finite-dimensional right $A$-modules and by $D$ the standard
duality $\Hom_K(-,K)$ on $\mod A$.
An algebra $A$ is called \emph{self-injective}
if $A_A$ is injective in $\mod A$, or equivalently,
the projective modules in $\mod A$ are injective.
A prominent class of self-injective algebras is formed
by the \emph{symmetric algebras} $A$ for which there exists
an associative, non-degenerate symmetric $K$-bilinear form
$(-,-): A \times A \to K$.
Classical examples of symmetric algebras are provided
by the blocks of group algebras of finite groups and
the Hecke algebras of finite Coxeter groups.
In fact, any algebra $A$ is a quotient algebra
of its trivial extension algebra $\T(A) = A \ltimes D(A)$,
which is a symmetric algebra.

For an algebra $A$, the module category $\mod A^e$ 
of its enveloping algebra $A^e = A^{\op} \otimes_K A$
is the category of finite-dimensional $A$-$A$-bimodules.
We denote by $\Omega_{A^e}$ the syzygy operator in $\mod A^e$
which assigns to a module $M$ in $\mod A^e$  the kernel 
$\Omega_{A^e}(M)$ of a minimal projective cover of $M$ in $\mod A^e$.
An algebra $A$ is called \emph{periodic}
if $\Omega_{A^e}^n(A) \cong A$ in $\mod A^e$ for some $n \geq 1$,
and if so the minimal such $n$ is called the period of $A$.
Periodic algebras are self-injective and have periodic
Hochschild cohomology.

Finding or possibly classifying periodic algebras
is an important problem. It is very interesting because of
connections
with group theory, topology, singularity theory
and cluster algebras.

We are concerned with the classification of all periodic tame
symmetric algebras.
In \cite{Du} Dugas proved that every representation-finite
self-injective algebra, without simple blocks, is a periodic algebra.
The representation-infinite, indecomposable,
periodic algebras of polynomial growth were classified
by Bia\l kowski, Erdmann and Skowro\'nski in \cite{BES}.
It is conjuctered in \cite[Problem]{ES2} that every indecomposable
symmetric periodic tame algebra of non-polynomial growth 
is of period $4$.
Prominent classes of tame symmetric algebras of period $4$
are provided by the weighted surface algebras and their
deformations investigated in
\cite{ES2},
\cite{ES3},
\cite{ES4},
\cite{ES5}.

In this article we introduce and study  higher spherical algebras,
which are  ``higher analogs'' of the non-singular spherical algebras
introduced in \cite{ES5},
and provide a new exotic family of tame symmetric periodic
algebras of period $4$.

Let $m \geq1$ be a natural number and  $\lambda \in K^*$.
We denote by $S(m,\lambda)$
the algebra given by the quiver $\Delta$
of the form
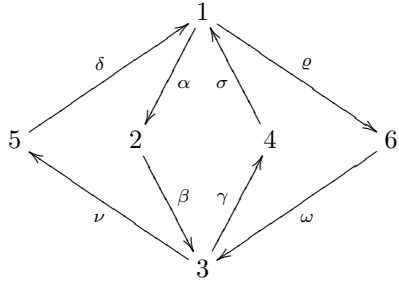
\[
  \xymatrix@R=3.pc@C=1.2pc{
    &&& 1
    \ar[ld]^{\alpha}
    \ar[rrrd]^{\varrho}
    \\   
    5
    \ar[rrru]^{\delta}
    && 2
    \ar[rd]^{\beta}
    && 4
    \ar[lu]^{\sigma}
    && 6
    \ar[llld]^{\omega}
    \\
   &&& 3
    \ar[lllu]^{\nu}
    \ar[ur]^{\gamma}
  }
\]
and the relations:
\begin{align*}
  \beta \nu \delta &= \beta \gamma \sigma
     + \lambda (\beta \gamma \sigma \alpha)^{m-1}  \beta \gamma \sigma ,
  &
  \alpha \beta \nu &= \varrho \omega \nu ,
\\
  \nu \delta \alpha &= \gamma \sigma \alpha
     + \lambda (\gamma \sigma \alpha \beta)^{m-1} \gamma \sigma \alpha ,
  &
  \delta \alpha \beta &=  \delta \varrho \omega ,
\\
  \sigma \varrho \omega &= \sigma \alpha \beta
    + \lambda (\sigma \alpha \beta \gamma)^{m-1} \sigma \alpha \beta ,
  &
  \omega \gamma \sigma &= \omega \nu \delta ,
\\
  \varrho \omega \gamma &= \alpha \beta \gamma
    + \lambda (\alpha \beta \gamma \sigma)^{m-1} \alpha \beta \gamma ,
  &
  \gamma \sigma \varrho &=  \nu \delta \varrho ,
\\
  &(\alpha \beta \gamma \sigma)^m \alpha = 0 ,
  &
  (\gamma \sigma \alpha \beta)^m \gamma &=  0 .
\end{align*}
We call $S(m,\lambda)$ with $m \geq2$
a \emph{higher spherical algebra}.
For  $m=1$, this 
is the non-singular spherical algebra $S(1+\lambda)$
investigated in \cite[Section~3]{ES5}.
The above quiver is its Gabriel quiver, 
and $S(1+\lambda)$ is a surface algebra (in the sense of \cite{ES5})
given by the following triangulation of the sphere $S^2$ in $\bR^3$
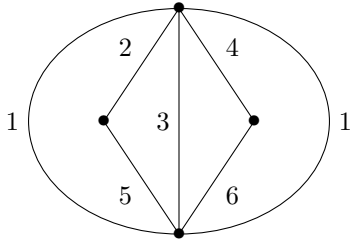
\[
\begin{tikzpicture}[xscale=2,yscale=1.5,auto]
\coordinate (o) at (0,0);
\coordinate (a) at (0,1);
\coordinate (b) at (-.5,0);
\coordinate (c) at (.5,0);
\coordinate (d) at (0,-1);
\coordinate (l) at (-1,0);
\coordinate (r) at (1,0);
\draw (a) to node {4} (c)
(c) to node {6} (d)
(d) to node {5} (b)
(b) to node {2} (a);
\draw (d) to node {3} (a);
\node (a) at (a) {$\bullet$};
\node (b) at (b) {$\bullet$};
\node (c) at (c) {$\bullet$};
\node (d) at (d) {$\bullet$};
\draw (l) arc (-180:180:1) node [left] {$1$};
\draw (r) node [right] {$1$};
\end{tikzpicture}
\]
with the coherent orientation of triangles:
\mbox{(1 2 5)}, 
\mbox{(2 3 5)}, 
\mbox{(3 4 6)},
\mbox{(4 1 6)}.
We note that the non-singular spherical algebras in  \cite{ES5} appear 
since in the general setting for weighted surface
algebras we allow `virtual' arrows. 

The following two theorems describe basic properties
of higher spherical algebras.

\begin{main-theorem}
\label{th:main1}
Let $S = S(m,\lambda)$ be a higher spherical algebra.
Then $S$ is a finite-dimensional algebra with 
$\dim_K S = 36m + 4$.
\end{main-theorem}

\begin{main-theorem}
\label{th:main2}
Let $S = S(m,\lambda)$ be a higher spherical algebra.
Then the following statements hold:
\begin{enumerate}[(i)]
 \item
  $S$ is a symmetric algebra.
 \item
  $S$ is a periodic algebra of period $4$.
 \item
  $S$ is a tame algebra of non-polynomial growth.
\end{enumerate}
\end{main-theorem}

It follows from the above theorems that 
the higher spherical algebras $S(m,\lambda)$,
$m \geq2$,  $\lambda \in K^*$, form
an exotic family of algebras of generalized quaternion type
(in the sense of \cite{ES4})
whose Gabriel quiver is not $2$-regular.
The classification of the Morita equivalence classes of all
algebras of generalized quaternion type with $2$-regular
Gabriel quivers having at least three vertices has been 
established in \cite[Main~Theorem]{ES4}.
During the work on this, surprisingly, we discovered 
new algebras, which we call higher tetrahedral algebras
$\Lambda(m,\lambda)$, $m \geq2$,  $\lambda \in K^*$, 
They are  introduced and  studied in \cite{ES3}
(see Section~\ref{sec:hta} for definition and properties).

The following theorem relates these two classes of algebras.

\begin{main-theorem}
\label{th:main3}
Let $m \geq2$ be a natural number and  $\lambda \in K^*$.
Then the algebras $S(m,\lambda)$ and $\Lambda(m,\lambda)$ 
are derived equivalent.
\end{main-theorem}

Then Theorem~\ref{th:main2} is the consequence of 
Theorem~\ref{th:main3}, by applying general theory as described in 
Theorems \ref{th:2.3}, \ref{th:2.4}, \ref{th:2.5},
and Theorem~\ref{th:3.1}.

For general background on the relevant representation theory
we refer to the books
\cite{ASS},
\cite{Ha},
\cite{SS},
\cite{SY}.

\section{Derived equivalences}
\label{sec:deriveq}

In this section we collect some facts on derived equivalences
of algebras which are needed in the proofs of
Theorems \ref{th:main2} and \ref{th:main3}.

Let $A$ be an algebra over $K$.
We denote by $D^b(\mod A)$ the derived category of $\mod A$,
which is  the localization of the homotopy category $K^b(\mod A)$
of bounded complexes of modules from $\mod A$ with respect
to quasi-isomorphisms.
Moreover, let $K^b(P_A)$ be the subcategory of $K^b(\mod A)$
given by the complexes of projective modules in $\mod A$.
Two algebras $A$ and $B$ are called \emph{derived equivalent}
if the derived categories 
$D^b(\mod A)$ and $D^b(\mod B)$
are equivalent as triangulated categories. The triangulated
structure is induced by shift in degrees of complexes. 
Following J.~Rickard \cite{Ric1}, 
a complex $T$ in  $K^b(P_A)$
is called a \emph{tilting complex} if the following properties are satisfied:
\begin{enumerate}[(1)]
 \item
  $\Hom_{K^b(P_A)}(T,T[i]) = 0$ for all $i \neq 0$ in $\bZ$,
 \item
  the full subcategory $\add(T)$ of $K^b(P_A)$ consisting
  of direct summands of direct sums of copies of $T$
  generates $K^b(P_A)$ as a triangulated category.
\end{enumerate}
Here, $[\,]$ denotes the translation functor by shifting any complex
one degree to the left.

The following theorem is due to J.~Rickard \cite[Theorem~6.4]{Ric1}.

\begin{theorem}
\label{th:2.1}
Two algebras $A$ and $B$ are derived equivalent if and only if
there is a tilting complex $T$ in $K^b(P_A)$ such that
$B \cong \End_{K^b(P_A)}(T)$.
\end{theorem}

We will need the following special case of an alternating
sum formula established by D.~Happel in 
\cite[Sections III.1.3 and III.1.4]{Ha}.

\begin{proposition}
\label{prop:2.2}
Let $A$ be an algebra and
$Q = (Q^r)_{r\in\bZ}$,
$R = (R^s)_{s\in\bZ}$
two complexes in $K^b(P_A)$ such that
$\Hom_{K^b(P_A)}(Q,R[i]) = 0$ for any $i \neq 0$ in $\bZ$.
Then
\[
   \dim_K \Hom_{K^b(P_A)}(Q,R) =  \sum_{r,s}(-1)^{r-s} \dim_K \Hom_{A}(Q^r,R^s) . 
\]
\end{proposition}

We note that the right-hand side of the above formula
can easily be computed using the Cartan matrix of $A$.

We end this section with the following collection
of important results.

\begin{theorem}
\label{th:2.3}
Let $A$ and $B$ be derived equivalent algebras.
Then $A$ is symmetric if and only if $B$ is symmetric.
\end{theorem}

\begin{proof}
This is \cite[Corollary~5.3]{Ric3}.
\end{proof}

\begin{theorem}
\label{th:2.4}
Let $A$ and $B$ be derived equivalent algebras.
Then $A$ is periodic  if and only if $B$ is periodic. 
Moreover, if so,  then they have the same period.
\end{theorem}

\begin{proof}
See \cite[Theorem~2.9]{ES1}.
\end{proof}

\begin{theorem}
\label{th:2.5}
Let $A$ and $B$ be derived equivalent selfinjective algebras.
Then the following equivalences hold.
\begin{enumerate}[(i)]
 \item
  $A$ is tame if and only if $B$ is tame.
 \item
  $A$ is of polynomial growth if and only if $B$ is of polynomial growth.
\end{enumerate}
\end{theorem}

\begin{proof}
It follows from the assumption and
\cite[Corollary~2.2]{Ric2}
(see also \cite[Corollary~5.3]{Ric3})
that the algebras $A$ and $B$ are stably equivalent.
Then the equivalences (i) and (ii) hold by 
\cite[Theorems 4.4 and 5.6]{CB}
and
\cite[Corollary~2]{KZ}.
\end{proof}

\section{Higher tetrahedral algebras}
\label{sec:hta}

In this section we recall some facts on higher tetrahedral
algebras established in \cite{ES3}, which will be crucial 
in the proofs
of Theorems \ref{th:main1} and \ref{th:main2}.

Consider the tetrahedron
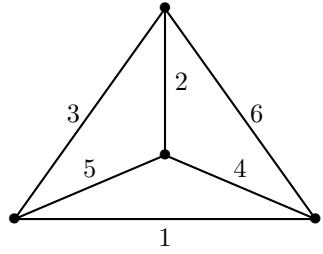
\[
\begin{tikzpicture}
[scale=1]
\node (A) at (-2,0) {$\bullet$};
\node (B) at (2,0) {$\bullet$};
\node (C) at (0,.85) {$\bullet$};
\node (D) at (0,2.8) {$\bullet$};
\coordinate (A) at (-2,0) ;
\coordinate (B) at (2,0) ;
\coordinate (C) at (0,.85) ;
\coordinate (D) at (0,2.8) ;
\draw[thick]
(A) edge node [left] {3} (D)
(D) edge node [right] {6} (B)
(D) edge node [right] {2} (C)
(A) edge node [above] {5} (C)
(C) edge node [above] {4} (B)
(A) edge node [below] {1} (B) ;
\end{tikzpicture}
\]
with the coherent orientation of triangles:
$(1\ 5\ 4)$,
$(2\ 5\ 3)$,
$(2\ 6\ 4)$,
$(1\ 6\ 3)$.
Then, following \cite[Section~6]{ES2},
we have the associated triangulation quiver
$(Q,f)$ of the form
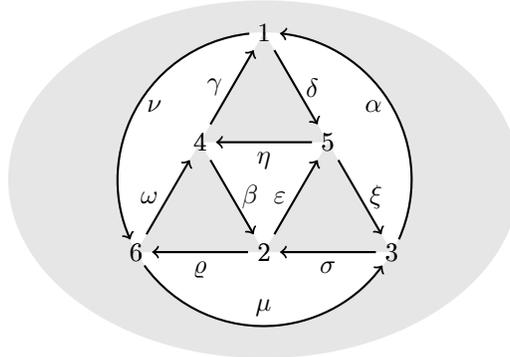
\[
\begin{tikzpicture}
[scale=.85]
\node (1) at (0,1.72) {$1$};
\node (2) at (0,-1.72) {$2$};
\node (3) at (2,-1.72) {$3$};
\node (4) at (-1,0) {$4$};
\node (5) at (1,0) {$5$};
\node (6) at (-2,-1.72) {$6$};
\coordinate (1) at (0,1.72);
\coordinate (2) at (0,-1.72);
\coordinate (3) at (2,-1.72);
\coordinate (4) at (-1,0);
\coordinate (5) at (1,0);
\coordinate (6) at (-2,-1.72);
\fill[fill=gray!20]
      (0,2.22cm) arc [start angle=90, delta angle=-360, x radius=4cm, y radius=2.8cm]
 --  (0,1.72cm) arc [start angle=90, delta angle=360, radius=2.3cm]
     -- cycle;
\fill[fill=gray!20]
    (1) -- (4) -- (5) -- cycle;
\fill[fill=gray!20]
    (2) -- (4) -- (6) -- cycle;
\fill[fill=gray!20]
    (2) -- (3) -- (5) -- cycle;

\node (1) at (0,1.72) {$1$};
\node (2) at (0,-1.72) {$2$};
\node (3) at (2,-1.72) {$3$};
\node (4) at (-1,0) {$4$};
\node (5) at (1,0) {$5$};
\node (6) at (-2,-1.72) {$6$};
\draw[->,thick] (-.23,1.7) arc [start angle=96, delta angle=108, radius=2.3cm] node[midway,right] {$\nu$};
\draw[->,thick] (-1.87,-1.93) arc [start angle=-144, delta angle=108, radius=2.3cm] node[midway,above] {$\mu$};
\draw[->,thick] (2.11,-1.52) arc [start angle=-24, delta angle=108, radius=2.3cm] node[midway,left] {$\alpha$};
\draw[->,thick]
 (1) edge node [right] {$\delta$} (5)
(2) edge node [left] {$\varepsilon$} (5)
(2) edge node [below] {$\varrho$} (6)
(3) edge node [below] {$\sigma$} (2)
 (4) edge node [left] {$\gamma$} (1)
(4) edge node [right] {$\beta$} (2)
(5) edge node [right] {$\xi$} (3)
 (5) edge node [below] {$\eta$} (4)
(6) edge node [left] {$\omega$} (4)
;
\end{tikzpicture}
\]
where $f$ is the permutation of arrows of order $3$ described
by the four shaded 
$3$-cycles.
We denote by $g$ the permutation on the set of arrows of $Q$ 
whose $g$-orbits are the 
four white 
$3$-cycles.
%

Let $m \geq2$ be a natural number and  $\lambda \in K^*$.
Following \cite{ES3},
the (non-singular) 
\emph{tetrahedral algebra $\Lambda(m,\lambda)$ of degree $m$} 
is the algebra given by the above
quiver $Q$ and the relations:
\begin{align*}
 \gamma\delta &= \beta\varepsilon + \lambda (\beta\varrho\omega)^{m-1} \beta\varepsilon, 
 &
 \delta\eta &= \nu\omega,
 &
 \eta\gamma &= \xi\alpha,
 & 
 \nu \mu &= \delta\xi ,
\\
 \varrho\omega &= \varepsilon\eta + \lambda (\varepsilon\xi\sigma)^{m-1} \varepsilon\eta,
 &
 \omega\beta &= \mu\sigma,
 &
 \beta\varrho &= \gamma\nu ,
 &
 \mu \alpha &= \omega \gamma ,
\\
 \xi\sigma &= \eta\beta + \lambda (\eta\gamma\delta)^{m-1} \eta\beta,
 &
 \sigma\varepsilon &= \alpha\delta,
 &
 \varepsilon\xi &= \varrho\mu,
 &
 \alpha\nu &= \sigma\varrho, 
\\
\omit\rlap{\qquad\quad\ \ $\big(\theta f(\theta) f^2(\theta)\big)^{m-1} \theta f(\theta) g\big(f(\theta)\big) = 0$ for any arrow $\theta$ in $Q$.}
\end{align*}

The following theorem follows from 
Theorems 1, 2, 3
proved in \cite{ES3} and describes some basic properties
of higher tetrahedral algebras.

\begin{theorem}
\label{th:3.1}
Let $\Lambda = \Lambda(m,\lambda)$ be a higher tetrahedral algebra
with $m \geq2$ and  $\lambda \in K^*$.
Then the following statements hold:
\begin{enumerate}[(i)]
 \item
  $\Lambda$ is a finite-dimensional algebra 
  with $\dim_K \Lambda = 36 m$.
 \item
  $\Lambda$ is a symmetric algebra.
 \item
  $\Lambda$ is a periodic algebra of period $4$.
 \item
  $\Lambda$ is a tame algebra of non-polynomial growth.
\end{enumerate}
\end{theorem}

The following proposition follows from 
\cite[Section~4]{ES3}.

\begin{proposition}
\label{prop:3.2}
Let $\Lambda = \Lambda(m,\lambda)$ be a higher tetrahedral algebra
with $m \geq2$ and  $\lambda \in K^*$.
Then the Cartan matrix $C_{\Lambda}$ of $\Lambda$ is of the form
\[
 \begin{bmatrix}
  m+1 & m-1 & m & m & m & m \\
  m-1 & m+1 & m & m & m & m \\
  m & m & m+1 & m-1 & m & m \\
  m & m & m-1 & m+1 & m & m \\
  m & m & m & m & m+1 & m-1 \\
  m & m & m & m & m-1 & m+1 
 \end{bmatrix} .
\]
\end{proposition}

\section{Proof of Theorem~\ref{th:main1}}\label{sec:proof1}

In this section we describe the Cartan matrices of 
higher spherical algebras.

Let $S = S(m,\lambda)$ be a higher spherical algebra
with $m \geq2$ and  $\lambda \in K^*$.
We start by collecting   further identities in $S$, they follow directly
from the relations defining $S$.

\begin{lemma}
\label{lem:4.1}
The following relations hold in $S$:
\begin{enumerate}[(i)]
 \item
  $(\beta \gamma \sigma \alpha)^{m-1}  \beta \gamma \sigma \varrho = 0$ and
  $(\alpha \beta \gamma \sigma)^m \varrho = 0$.
 \item
  $\omega \gamma \sigma \alpha (\beta \gamma \sigma \alpha)^{m-1} = 0$ and
  $\omega (\gamma \sigma \alpha \beta)^m = 0$.
 \item
  $(\sigma \alpha \beta \gamma)^{m-1} \sigma \alpha \beta \nu = 0$ and 
  $(\gamma \sigma \alpha \beta)^m \nu = 0$.
 \item
  $\delta \alpha \beta \gamma (\sigma \alpha \beta \gamma)^{m-1} = 0$ and
  $\delta (\alpha \beta \gamma \sigma)^{m} = 0$.
 \item
  $(\varrho \omega \nu \delta)^r = (\alpha \beta \gamma \sigma)^r$ and 
  $(\nu \delta \varrho \omega)^r = (\gamma \sigma \alpha \beta)^r$,
  for  $2 \leq r \leq m$.
\end{enumerate}
\end{lemma}

\bigskip

For example, consider part  (i). We have
$\lambda(\beta\gamma\sigma\alpha)^{m-1}\beta\gamma\sigma = \beta\nu\delta - \beta\gamma\sigma.
$
We postmultiply this with $\rho$ and get zero since $\gamma\sigma\rho = \nu\delta\rho$. The second part follows, by rewriting the first part and premultiply with $\alpha$. 
For part (v), starting with
$$\rho\omega\nu\delta = \alpha\beta\gamma\sigma + \lambda(\alpha\beta\gamma\sigma)^m
$$
and squaring, one gets  
$(\rho\omega\nu\delta)^2 = (\alpha\beta\gamma\sigma)^2$ and
(v) follows by induction.

Using the relations, it is easy to  write down bases for the 
indecomposable projective modules, and 
prove the following.

\bigskip

\begin{proposition}
\label{prop:4.2}
The Cartan matrix $C_{S}$ of $S$ is of the form
\[
 \begin{bmatrix}
  m+1 & m & m+1 & m & m & m \\
  m & m+1 & m & m & m & m-1 \\
  m+1 & m & m+1 & m & m & m \\
  m & m & m & m+1 & m-1 & m \\
  m & m & m & m-1 & m+1 & m \\
  m & m-1 & m & m & m & m+1 
 \end{bmatrix} .
\]
In particular, $\dim_K S = 36m+4$.
\end{proposition}

\section{Proof of Theorem~\ref{th:main3}}\label{sec:proof3}

Let $\Lambda = \Lambda(m,\lambda)$ 
for some fixed
$m \geq2$ and  $\lambda \in K^*$.
For each vertex $i$ of the quiver $Q$ defining $\Lambda$,
we denote by $P_i = e_i \Lambda$ the associated 
indecomposable projective module in $\mod \Lambda$.
Moreover, for any arrow $\theta$ from $j$ to $k$,
we denote by
$\theta : P_k \to  P_j$ the homomorphism in $\mod \Lambda$
given by the left multiplication by $\theta$.
We consider the following complexes in $K^b(P_{\Lambda})$:
\begin{align*}
 T_1 &: \quad 0 \longrightarrow P_1 \longrightarrow 0, \\
 T_2 &: \quad 0 \longrightarrow P_5 \longrightarrow 0, \\
 T_4 &: \quad 0 \longrightarrow P_3 \longrightarrow 0, \\
 T_5 &: \quad 0 \longrightarrow P_4 \longrightarrow 0, \\
 T_6 &: \quad 0 \longrightarrow P_6 \longrightarrow 0,
\end{align*}
concentrated in degree $0$, and
\begin{align*}
 T_3 &: \quad 0 \longrightarrow 
      P_2 \xrightarrow{\left[\begin{smallmatrix} -\sigma\\\beta\end{smallmatrix}\right]} 
      P_3 \oplus P_4 \longrightarrow 0, 
\end{align*}
concentrated in degrees $1$ and $0$.
Moreover, we set
\[
  T = T_1 \oplus T_2 \oplus T_3 \oplus T_4 \oplus T_5 \oplus T_6 .
\]

\begin{lemma}
\label{lem:5.1}
$T$ is a tilting complex in $K^b(P_{\Lambda})$.
\end{lemma}

\begin{proof}
It is sufficient to show the equalities
\begin{align*}
  \Hom_{K^b(P_{\Lambda})} (T_3, T_r[1]) &= 0 , & 
  \Hom_{K^b(P_{\Lambda})} (T_r, T_3[-1]) &= 0,  
\end{align*}
for $r \in \{1,2,\dots,6\}$.
The first equalities hold, 
because any nonzero homomorphism
$f : P_2 \to P_i$ with $i \neq 2$ factors
through 
$\left[\begin{smallmatrix} -\sigma\\\beta\end{smallmatrix}\right] : P_2 \to  P_3 \oplus P_4$.
The second equalities hold, because for any nonzero
$g : P_i \to P_2$ with $i \neq 2$, the composition 
$\left[\begin{smallmatrix} -\sigma\\\beta\end{smallmatrix}\right] g$
is nonzero.
\end{proof}

We define $R = R(m,\lambda) = \End_{K^b(P_{\Lambda})} (T)$,
and note that
$\tilde{P}_i = \Hom_{K^b(P_{\Lambda})} (T, T_i)$, $i \in \{1,2,\dots,6\}$,
form a complete family of pairwise non-isomorphic indecomposable
projective modules in $\mod R$.
We abbreviate $S = S(m,\lambda)$,
and use the ordering $\bar{P}_i = e_i S$, $i \in \{1,2,\dots,6\}$,
of the indecomposable projective modules in $\mod S$ 
corresponding to the numbering  of vertices of the quiver
$\Delta$ defining $S$.
In this notation, we have the following lemma.

\begin{lemma}
\label{lem:5.2}
The Cartan matrices $C_R$ and $C_S$  coincide.
In particular, the algebras $R$ and $S$ have the same dimension $36m+4$.
\end{lemma}

\begin{proof}
This follows by the computation of the dimensions
$\dim_K \Hom_{K^b(P_{\Lambda})} (T_i, T_j)$, $i,j \in \{1,2,\dots,6\}$,
using Proposition~\ref{prop:2.2}
and the form of the Cartan matrix $C_{\Lambda}$ of $\Lambda$
presented in Proposition~\ref{prop:3.2}.
For example, the first row of $C_R$ is of the form
$[m+1 \quad m \quad m+1 \quad m \quad m \quad m \quad m]$,
because
\begin{align*}
  \dim_K \Hom_{K^b(P_{\Lambda})} (T_1, T_1) 
    &= \dim_K \Hom_{\Lambda} (P_1, P_1) = m+1, \\ 
  \dim_K \Hom_{K^b(P_{\Lambda})} (T_1, T_2) 
    &= \dim_K \Hom_{\Lambda} (P_1, P_5) = m, \\
  \dim_K \Hom_{K^b(P_{\Lambda})} (T_1, T_3) 
    &= \dim_K \Hom_{\Lambda} (P_1, P_3) 
        + \dim_K \Hom_{\Lambda} (P_1, P_4) 
 \\&\, \quad
        - \dim_K \Hom_{\Lambda} (P_1, P_2) 
 \\&
       = m + m - (m-1)
       = m+1, \\ 
  \dim_K \Hom_{K^b(P_{\Lambda})} (T_1, T_4) 
    &= \dim_K \Hom_{\Lambda} (P_1, P_3) = m, \\
  \dim_K \Hom_{K^b(P_{\Lambda})} (T_1, T_5) 
    &= \dim_K \Hom_{\Lambda} (P_1, P_4) = m, \\
  \dim_K \Hom_{K^b(P_{\Lambda})} (T_1, T_6) 
    &= \dim_K \Hom_{\Lambda} (P_1, P_6) = m.
\end{align*}
\end{proof}

We define now irreducible morphisms between the summands
of $T$ in $K^b(P_{\Lambda})$:
\begin{align*}
  \tilde{\alpha} &: T_2 \to T_1, \mbox{ given by } \delta : P_5 \to P_1, \\
  \tilde{\beta} &: T_3 \to T_2, \mbox{ given by } 
         [\xi\quad \eta + \lambda(\eta \gamma\delta)^{m-1} \eta] : P_3 \oplus P_4 \to P_5, \\
  \tilde{\gamma} &: T_4 \to T_3, \mbox{ given by } 
     \left[\begin{smallmatrix} 1\\0\end{smallmatrix}\right] : 
      P_3 \to  P_3 \oplus P_4, \\
  \tilde{\sigma} &: T_1 \to T_4, \mbox{ given by } \alpha : P_1 \to P_3, \\
  \tilde{\delta} &: T_1 \to T_5, \mbox{ given by } \gamma : P_1 \to P_4, \\
  \tilde{\varrho} &: T_6 \to T_1, \mbox{ given by } \nu : P_6 \to P_1, \\
  \tilde{\omega} &: T_3 \to T_6, \mbox{ given by } 
         [\mu\quad \omega] : P_3 \oplus P_4 \to P_6, \\
  \tilde{\nu} &: T_5 \to T_3, \mbox{ given by } 
     \left[\begin{smallmatrix} 0\\1\end{smallmatrix}\right] : 
      P_4 \to  P_3 \oplus P_4.
\end{align*}
We obtain then the irreducible homomorphisms between 
the indecomposable projective modules in $\mod R$
\begin{align*}
  \alpha &= \Hom_{K^b(P_{\Lambda})} (T, \tilde{\alpha}) : \tilde{P}_2 \to \tilde{P}_1 , & 
  \beta &= \Hom_{K^b(P_{\Lambda})} (T,  \tilde{\beta}) : \tilde{P}_3 \to \tilde{P}_2 , \\
  \gamma &= \Hom_{K^b(P_{\Lambda})} (T, \tilde{\gamma}) : \tilde{P}_4 \to \tilde{P}_3 , & 
  \sigma &= \Hom_{K^b(P_{\Lambda})} (T,  \tilde{\sigma}) : \tilde{P}_1 \to \tilde{P}_4 , \\
  \delta &= \Hom_{K^b(P_{\Lambda})} (T, \tilde{\delta}) : \tilde{P}_1 \to \tilde{P}_4 , & 
  \varrho &= \Hom_{K^b(P_{\Lambda})} (T,  \tilde{\varrho}) : \tilde{P}_6 \to \tilde{P}_1 , \\
  \omega &= \Hom_{K^b(P_{\Lambda})} (T, \tilde{\omega}) : \tilde{P}_3 \to \tilde{P}_6 , & 
  \nu &= \Hom_{K^b(P_{\Lambda})} (T,  \tilde{\nu}) : \tilde{P}_5 \to \tilde{P}_3 , 
\end{align*}
which are representatives of all irreducible homomorphisms
between the modules $\tilde{P}_i$, $i \in \{1,2,\dots,6\}$,
in $\mod R$.
This shows that the Gabriel quiver $Q_R$ of $R$ is the quiver
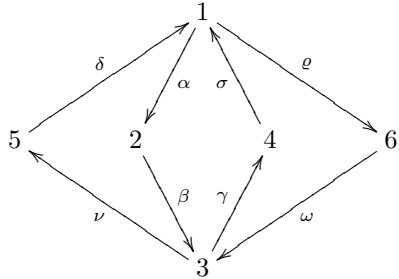
\[
  \xymatrix@R=3.pc@C=1.2pc{
    &&& 1
    \ar[ld]^{\alpha}
    \ar[rrrd]^{\varrho}
    \\   
    5
    \ar[rrru]^{\delta}
    && 2
    \ar[rd]^{\beta}
    && 4
    \ar[lu]^{\sigma}
    && 6
    \ar[llld]^{\omega}
    \\
   &&& 3
    \ar[lllu]^{\nu}
    \ar[ur]^{\gamma}
  }
\]
being the quiver $\Delta$ defining the algebra $S$.

\begin{theorem}
\label{th:5.3}
The algebras $R$ and $S$ are isomorphic.
\end{theorem}

\begin{proof}
We first prove that the following identities  hold in $R$:
\begin{align}
 \label{eq:1}
  \alpha\beta\gamma &= \varrho\omega\gamma , \\
 \label{eq:2}
  \sigma\alpha\beta &= \sigma\varrho\omega , \\
 \label{eq:3}
  \gamma \sigma \varrho &=  \nu \delta \varrho , \\
 \label{eq:4}
  \omega \gamma \sigma &= \omega \nu \delta , \\
 \label{eq:5}
  \beta \nu \delta &= \beta \gamma \sigma
     + \lambda (\beta \gamma \sigma \alpha)^{m-1}  \beta \gamma \sigma , \\
 \label{eq:6}
  \nu \delta \alpha &= \gamma \sigma \alpha
     + \lambda (\gamma \sigma \alpha \beta)^{m-1} \gamma \sigma \alpha , \\
 \label{eq:7}
  \delta \alpha \beta &= \delta \varrho \omega
    + \lambda (\delta \varrho \omega \nu)^{m-1} \delta \varrho \omega , \\
 \label{eq:8}
  \alpha \beta \nu &= \varrho \omega \nu
    + \lambda (\varrho \omega \nu \delta)^{m-1} \varrho \omega \nu , \\
 \label{eq:9}
  (\alpha \beta \gamma \sigma)^m \alpha &= 0 , \\
 \label{eq:10}
  (\gamma \sigma \alpha \beta)^m \gamma &=  0 . 
\end{align}

For \eqref{eq:1}, it is enough to show that
$\tilde{\alpha}\tilde{\beta}\tilde{\gamma} = \tilde{\varrho}\tilde{\omega}\tilde{\gamma}$.
We have 
$\tilde{\alpha}\tilde{\beta}\tilde{\gamma} = \delta \xi : P_3 \to P_1$
and
$\tilde{\varrho}\tilde{\omega}\tilde{\gamma} = \nu\mu : P_3 \to P_1$,
with $\delta \xi = \nu\mu$ in $\Lambda$, 
and so the required equality holds.

For \eqref{eq:2}, it is enough to show that
$\tilde{\sigma}\tilde{\alpha}\tilde{\beta} = \tilde{\sigma}\tilde{\varrho}\tilde{\omega}$.
We have 
$\tilde{\sigma}\tilde{\alpha}\tilde{\beta}
 = [\alpha\delta\xi\quad
   \alpha\delta\eta + \lambda\alpha\delta(\eta \gamma\delta)^{m-1} \eta]
  : P_3 \oplus P_4 \to P_3$
and
$\tilde{\sigma}\tilde{\varrho}\tilde{\omega}
  = [\alpha\nu\mu, \alpha\nu\omega] : P_3 \oplus P_4 \to P_3$.
Moreover, we have in $\Lambda$ the equalities
\[
  \alpha\nu\omega = \sigma\varrho\omega
  = \sigma \varepsilon \eta  + \lambda \sigma(\varepsilon \xi \sigma)^{m-1} \varepsilon \eta
  = \alpha \delta \eta  + \lambda \alpha \delta (\xi \sigma \varepsilon)^{m-1} \eta
\]
and
$\xi \sigma \varepsilon = \xi \alpha \delta = \eta \gamma \delta$ in $\Lambda$.
Hence the required equality holds.

For \eqref{eq:3}, we prove that
$\tilde{\nu}\tilde{\delta}\tilde{\varrho} - \tilde{\gamma}\tilde{\sigma}\tilde{\varrho} = 0$
in $K^b(P_{\Lambda})$.
We have
\[
 \tilde{\gamma}\tilde{\sigma}\tilde{\varrho} =
   \left[\begin{smallmatrix} \alpha\nu\\0\end{smallmatrix}\right] : 
      P_6 \to  P_3 \oplus P_4
 \quad
 \mbox{ and }
 \quad
 \tilde{\nu}\tilde{\delta}\tilde{\varrho} =
   \left[\begin{smallmatrix} 0\\\gamma\nu\end{smallmatrix}\right] : 
      P_6 \to  P_3 \oplus P_4.
\]
Moreover, we have the following commutative diagram in $\mod \Lambda$
\[
  \xymatrix@R=3.pc@C=4pc{
    & P_6 \ar[ld]_{\varrho}
    \ar[d]^{\left[\begin{smallmatrix} -\alpha\nu\\\gamma\nu\end{smallmatrix}\right]}
    \\   
    P_2 
    \ar[r]^{\left[\begin{smallmatrix} -\sigma\\\beta\end{smallmatrix}\right]}
    & P_3 \oplus P_4
  }
\]
because $\alpha\nu = \sigma\varrho$ and $\gamma\nu = \beta\varrho$ in $\Lambda$.
This proves the claim.

For \eqref{eq:4}, we note that 
$\tilde{\omega}\tilde{\gamma}\tilde{\sigma} = \mu\alpha : P_1 \to P_6$
and
$\tilde{\omega}\tilde{\nu}\tilde{\delta} = \omega\gamma : P_1 \to P_6$,
with $\mu\alpha = \omega\gamma$.

For \eqref{eq:5}, we prove equality
$\tilde{\beta}\tilde{\nu}\tilde{\delta} = \tilde{\beta}\tilde{\gamma}\tilde{\sigma}
  + \lambda (\tilde{\beta}\tilde{\gamma}\tilde{\sigma}\tilde{\alpha})^{m-1}
   \tilde{\beta}\tilde{\gamma}\tilde{\sigma}$.
Observe that,
\begin{align*}
 \tilde{\beta}\tilde{\nu}\tilde{\delta} 
  &= \eta\gamma + \lambda(\eta\gamma\delta)^{m-1} \eta\gamma 
   : P_1 \to P_5 , \\
 \tilde{\beta}\tilde{\gamma}\tilde{\sigma}
  &= \xi\alpha = \eta\gamma 
   : P_1 \to P_5 , \\
 \lambda (\tilde{\beta}\tilde{\gamma}\tilde{\sigma}\tilde{\alpha})^{m-1}
   \tilde{\beta}\tilde{\gamma}\tilde{\sigma}
 &= \lambda (\xi\alpha\delta)^{m-1}\eta\gamma
 = \lambda (\eta\gamma\delta)^{m-1}\eta\gamma : P_1 \to P_5,
\end{align*}
and hence the required equality holds.

For \eqref{eq:6}, we prove that
$\tilde{\nu}\tilde{\delta}\tilde{\alpha} = \tilde{\gamma}\tilde{\sigma}\tilde{\alpha}
  + \lambda (\tilde{\gamma}\tilde{\sigma}\tilde{\alpha}\tilde{\beta})^{m-1}
   \tilde{\gamma}\tilde{\sigma}\tilde{\alpha}$
in $K^b(P_{\Lambda})$.
We first observe that
\begin{align*}
 \tilde{\nu}\tilde{\delta}\tilde{\alpha}
  &= \left[\begin{smallmatrix} 0\\\gamma\delta\end{smallmatrix}\right]
  = \left[\begin{smallmatrix}
        0\\\beta\varepsilon+\lambda(\beta\varrho\omega)^{m-1}\beta\varepsilon
     \end{smallmatrix}\right]
   : P_5 \to P_3 \oplus P_4 , \\
\tilde{\gamma}\tilde{\sigma}\tilde{\alpha}
  &= \left[\begin{smallmatrix} \alpha\delta\\0\end{smallmatrix}\right]
   : P_5 \to P_3 \oplus P_4 , \\
\tilde{\gamma}\tilde{\sigma}\tilde{\alpha}\tilde{\beta}
  &= \left[\begin{smallmatrix} 
     \alpha\delta\xi & 
     \alpha\delta\eta + \lambda \alpha \delta(\eta\gamma\delta)^{m-1}\eta \\ 
    0 & 0
    \end{smallmatrix}\right]
   : P_3 \oplus P_4 \to P_3 \oplus P_4 .
\end{align*}
Moreover, 
$\alpha\delta = \sigma\varepsilon$,
and hence
\[
  \lambda (\alpha\delta\xi)^{m-1}\alpha\delta 
  = \lambda(\sigma\varepsilon\xi)^{m-1}\sigma\varepsilon
  = \lambda\sigma(\varepsilon\xi\sigma)^{m-1}\varepsilon .
\]
But then
\[
  \tilde{\gamma}\tilde{\sigma}\tilde{\alpha}
  + \lambda (\tilde{\gamma}\tilde{\sigma}\tilde{\alpha}\tilde{\beta})^{m-1}
   \tilde{\gamma}\tilde{\sigma}\tilde{\alpha}
  = \left[\begin{smallmatrix}
        \sigma\varepsilon+\lambda\sigma(\varepsilon\xi\sigma)^{m-1}\varepsilon \\ 0
     \end{smallmatrix}\right]
   : P_5 \to P_3 \oplus P_4 , \\
\]
Further,
we have the following commutative diagram in $\mod \Lambda$
\[
  \xymatrix@R=3.pc@C=4pc{
    & P_5 \ar[ld]_{\varepsilon+\lambda(\varepsilon\xi\sigma)^{m-1}\varepsilon}
    \ar[d]^{\left[\begin{smallmatrix}
     -\sigma\varepsilon-\lambda\sigma(\varepsilon\xi\sigma)^{m-1}\varepsilon\\
     \beta\varepsilon+\lambda\beta(\varrho\omega\beta)^{m-1}\varepsilon
   \end{smallmatrix}\right]}
    \\   
    P_2 
    \ar[r]^{\left[\begin{smallmatrix} -\sigma\\\beta\end{smallmatrix}\right]}
    & P_3 \oplus P_4
  }
\]
because $\varrho\omega\beta = \varrho\mu\sigma = \varepsilon\xi\sigma$ in $\Lambda$.
This shows that
$\tilde{\nu}\tilde{\delta}\tilde{\alpha} - \tilde{\gamma}\tilde{\sigma}\tilde{\alpha}
  - \lambda (\tilde{\gamma}\tilde{\sigma}\tilde{\alpha}\tilde{\beta})^{m-1}
   \tilde{\gamma}\tilde{\sigma}\tilde{\alpha} = 0$
in $K^b(P_{\Lambda})$,
and hence the required equality holds.

For \eqref{eq:7}, we prove that
$\tilde{\delta}\tilde{\alpha}\tilde{\beta} = \tilde{\delta}\tilde{\varrho}\tilde{\omega}
  + \lambda (\tilde{\delta}\tilde{\varrho}\tilde{\omega}\tilde{\nu})^{m-1}
  \tilde{\delta}\tilde{\varrho}\tilde{\omega}$.
We have
\begin{align*}
 \tilde{\delta}\tilde{\alpha}\tilde{\beta}
  &= [\gamma\delta\xi \quad \gamma\delta\eta
     + \lambda \gamma\delta (\eta \gamma\delta)^{m-1}\eta]
   : P_3 \oplus P_4 \to P_4, \\
 \tilde{\delta}\tilde{\varrho}\tilde{\omega}
  &= [\gamma\nu\mu \quad \gamma\nu\omega]
   : P_3 \oplus P_4 \to P_4 ,
\end{align*}
and $\delta\xi  = \nu\mu$, $\delta\eta = \nu\omega$ in $\Lambda$.
Furthermore,
\[
  \tilde{\delta}\tilde{\varrho}\tilde{\omega}\tilde{\nu} 
   = \gamma\nu\omega
   = \gamma\delta\eta 
   : P_4 \to P_4 .
\]
Hence we obtain
\[
  \lambda (\tilde{\delta}\tilde{\varrho}\tilde{\omega}\tilde{\nu})^{m-1}
    \tilde{\delta}\tilde{\varrho}\tilde{\omega}
  = \big[\lambda (\gamma\delta\eta)^{m-1} \gamma\delta\xi \quad 
    \lambda (\gamma\delta\eta)^{m-1} \gamma\delta\eta\big]
   : P_3 \oplus P_4 \to P_4  .
\]
We note that
\[
  \lambda (\gamma\delta\eta)^{m-1} \gamma\delta\xi
  = \lambda \big(\gamma f(\gamma)f^2(\gamma)\big)^{m-1} \gamma f(\gamma)g\big(f(\gamma)\big)
  = 0 .
\]
Therefore, the required equality holds.

For \eqref{eq:8}, we have to show that
$\tilde{\alpha}\tilde{\beta}\tilde{\nu}  = \tilde{\varrho}\tilde{\omega}\tilde{\nu} 
  + \lambda (\tilde{\varrho}\tilde{\omega}\tilde{\nu}\tilde{\delta})^{m-1}
  \tilde{\varrho}\tilde{\omega}\tilde{\nu} $.
We have
\begin{align*}
 \tilde{\alpha}\tilde{\beta}\tilde{\nu}
  &= \delta\eta + \lambda \delta (\eta \gamma\delta)^{m-1}\eta
   = \delta\eta + \lambda (\delta \eta \gamma)^{m-1}\delta\eta
   : P_4 \to P_1, \\
 \tilde{\varrho}\tilde{\omega}\tilde{\nu}
  &= \nu\omega = \delta \eta
   : P_4 \to P_1 , \\
 \tilde{\varrho}\tilde{\omega}\tilde{\nu}\tilde{\delta}
  &= \delta \nu\omega = \delta \eta \gamma
   : P_1 \to P_1 ,
\end{align*}
and then
\[
  \lambda (\tilde{\varrho}\tilde{\omega}\tilde{\nu}\tilde{\delta})^{m-1}
    \tilde{\varrho}\tilde{\omega}\tilde{\nu}
  = \lambda (\delta\eta\gamma)^{m-1} \delta\eta
   : P_4 \to P_1  .
\]
Hence the required equality holds.

For \eqref{eq:9}, we observe that
\[
  (\tilde{\alpha}\tilde{\beta}\tilde{\gamma}\tilde{\sigma})^{m} \tilde{\alpha}
  = ({\delta}{\xi}{\alpha})^{m} {\delta}
     : P_5 \to P_1  ,
\]
and this is zero,  because in $\Lambda$, the element
$({\delta}{\xi}{\alpha})^{m}$
belongs to the socle of $P_1$
(see \cite[Lemma~4.2]{ES3}).

For \eqref{eq:10}, we observe that
$\tilde{\sigma}\tilde{\alpha}\tilde{\beta}\tilde{\gamma} = \alpha\delta\xi: P_3\to P_3$, and therefore
\[
  (\tilde{\gamma}\tilde{\sigma}\tilde{\alpha}\tilde{\beta})^{m} \tilde{\gamma}
  = \left[\begin{smallmatrix}
        (\alpha\delta \xi)^m \\ 0
     \end{smallmatrix}\right]
  = \left[\begin{smallmatrix}
        (\sigma\varepsilon\xi)^{m} \\ 0
     \end{smallmatrix}\right]
   : P_3 \to P_3 \oplus P_4 , 
\]
Then,
we have in $\mod \Lambda$ the 
commutative diagram
\[
  \xymatrix@R=3.pc@C=4pc{
    & P_3 \ar[ld]_{-\varepsilon\xi(\sigma\varepsilon\xi)^{m-1}}
    \ar[d]^{\left[\begin{smallmatrix} (\sigma\varepsilon\xi)^{m}\\0\end{smallmatrix}\right]}
    \\   
    P_2 
    \ar[r]^{\left[\begin{smallmatrix} -\sigma\\\beta\end{smallmatrix}\right]}
    & P_3 \oplus P_4
  }
\]
because $\beta\varepsilon\xi(\sigma\varepsilon\xi)^{m-1}$
is the path of length $3m$ in $Q$ from $4$ to $3$,
and hence the zero path in $\Lambda$,
by \cite[Lemma~4.5]{ES3}.
Therefore,
$(\tilde{\gamma}\tilde{\sigma}\tilde{\alpha}\tilde{\beta})^{m} \tilde{\gamma} = 0$
in $K^b(P_{\Lambda})$, 
and  equality \eqref{eq:10} holds.

We also observe that, in $R$, we have  by (1) and (4) that
$\alpha\beta\gamma\sigma = \varrho\omega\gamma\sigma = \varrho\omega\nu\delta$.

\medskip

To obtain the defining relations for $S$, we
replace  $\varrho$ by 
$\varrho^* = \varrho + \lambda(\alpha\beta\gamma\sigma)^{m-1} \varrho
  = \varrho + \lambda(\varrho \omega \nu \delta)^{m-1} \varrho$.
Then identities  
\eqref{eq:1},
\eqref{eq:2},
\eqref{eq:3},
\eqref{eq:7},
\eqref{eq:8}
are replaced by the following identities:
\begin{align}
 \label{eq:1p}
 \tag{1*}
  \varrho^*\omega\gamma &= \varrho\omega\gamma  
 + \lambda(\alpha\beta\gamma\sigma)^{m-1} \varrho\omega\gamma
  = \alpha\beta\gamma + \lambda(\alpha\beta\gamma\sigma)^{m-1}  \alpha\beta\gamma, \\
 \label{eq:2p}
 \tag{2*}
  \sigma\varrho^*\omega &= \sigma\varrho\omega 
   + \lambda\sigma(\alpha\beta\gamma\sigma)^{m-1} \varrho \omega 
  = \sigma\alpha\beta + \lambda(\sigma\alpha\beta\gamma)^{m-1}\sigma \alpha\beta, \\  
 \label{eq:3p}
 \tag{3*}
	\gamma\sigma\varrho^* &= 
 \gamma\sigma\varrho + \lambda(\gamma\sigma(\varrho\omega\nu\delta)^{m-1}\varrho)  
	= \gamma\sigma\varrho + \lambda(\gamma\sigma\varrho)(\omega\nu\sigma\varrho)^{m-1}\\&\nonumber 
	= \nu\delta\varrho + \lambda(\nu\delta(\varrho\omega\nu\delta)^{m-1}\varrho) 
	= \nu\delta\varrho + \lambda(\nu\delta\varrho)(\omega\nu\sigma\varrho)^{m-1} = \nu\delta\varrho^*, \\
	\tag{7*}
 \delta \varrho^* \omega
 &= \delta \big(\varrho
    + \lambda (\varrho \omega \nu \delta)^{m-1} \varrho \big) \omega
   = \delta \varrho \omega
    + \lambda \delta (\varrho \omega \nu \delta)^{m-1} \varrho \omega
 \\&
\nonumber   = \delta \varrho \omega
    + \lambda (\delta \varrho \omega \nu)^{m-1} \delta \varrho \omega
   = \delta \alpha \beta ,
\\
 \label{eq:8p}
 \tag{8*}
  \varrho^* \omega \nu & 
    = \big(\varrho +  \lambda (\varrho \omega \nu \delta)^{m-1} \varrho\big) \omega \nu
    = \varrho \omega \nu
        + \lambda (\varrho \omega \nu \delta)^{m-1} \varrho \omega \nu
    = \alpha \beta \nu .
\end{align}

Therefore, after  replacing $\varrho$ by $\varrho^*$,
the relations defining the algebra $S = S(m,\lambda)$ are satisfied.
Then, applying 
Lemma~\ref{lem:5.2},
we conclude that algebras  $R$ and $S$ are isomorphic.
\end{proof}

Summing up, Theorem~\ref{th:main3} follows from
Theorems \ref{th:2.1} and \ref{th:5.3}.

%
%



\end{document}